\documentclass[11pt,twoside]{article}
\usepackage{acta-m,euler}

\title{Imbalances in directed multigraphs}

\newcommand{\vol}{\textbf{2}, 2 (2010) 137--145}   
\newcommand{\amss}{\amssubj{05C20}}
\newcommand{\keyw}{\keywords{digraph, imbalance, outdegree, indegree, directed multigraph, arc.}}

\begin{document}
\maketitle

\twoauthors{S. Pirzada}{Department of Mathematics, \\ University
of Kashmir, Srinagar, India}{sdpirzada@yahoo.co.in}
 {T. A. Naikoo}{Department of Mathematics, \\ University of Kashmir, Srinagar, India}{tariqnaikoo@rediffmail.com}

\twoauthors{U. Samee}{Department of Mathematics, \\ University of Kashmir, Srinagar, India}{pzsamee@yahoo.co.in}
 {A. Iv\'anyi}{Department of Computer Algebra, E\"otv\"os Lor\'and University, Hungary}{tony@compalg.inf.elte.hu}

\short{S. Pirzada, T. A. Naikoo, U. Samee, A. Iv\'anyi}{Imbalances in directed multigraphs }

\begin{abstract}
 In a directed multigraph, the imbalance of a vertex $v_{i}$ is defined as $b_{v_{i}}=d_{v_{i}}^{+}-d_{v_{i}}^{-}$,
where $d_{v_{i}}^{+}$ and $d_{v_{i}}^{-}$ denote the outdegree and indegree respectively of $v_{i}$. We characterize imbalances in directed multigraphs
and obtain lower and upper bounds on imbalances in such digraphs. Also, we show the existence of a directed
multigraph with a given imbalance set.
\end{abstract}

\section{Introduction}
A directed graph (shortly digraph) without loops and without multi-arcs is called a simple digraph \cite{Gross2004}. The imbalance of a vertex
$v_i$ in a digraph as $b_{v_{i}}$ (or simply $b_{i})= d_{v_{i}}^{+}-d_{v_{i}}^{-}$, where $d_{v_{i}}^{+}$ and
$d_{v_{i}}^{-}$ are respectively the outdegree and indegree of $v_i$. The imbalance sequence of a simple digraph is formed by listing
the vertex imbalances in non-increasing order. A sequence of integers $F=[f_{1},f_{2},\ldots,f_{n}]$ with $f_{1}\geq f_{2}\geq\ldots\geq
f_{n}$ is feasible if the sum of its elements is zero, and satisfies
$\displaystyle\sum_{i=1}^{k}f_{i}\leq k(n-k),$
for $1\leq k< n$.\\

The following result \cite{2} provides a necessary and sufficient
condition for a sequence of integers to be the imbalance sequence of a simple digraph.\\

\begin{theorem}
A sequence is realizable as an imbalance sequence if and only if it is feasible.
\end{theorem}

The above result is equivalent to saying that a sequence of integers $B=[b_{1},b_{2},\ldots,b_{n}]$ with $b_{1}\geq
b_{2}\geq\ldots\geq b_{n}$ is an imbalance sequence of a simple digraph if and only if
\begin{equation*}
\sum^{k}_{i=1}b_i\leq k(n-k),
\end{equation*}
for $1\leq k< n$, with equality when $k=n$.

On arranging the imbalance sequence in non-decreasing order, we have the following observation.

\begin{corollary}
A sequence of integers   $B=[b_{1},b_{2},\ldots,b_{n}]$ with $b_{1}\leq b_{2}\leq\ldots\leq  b_{n}$ is an imbalance sequence of a simple digraph if
and only if
\begin{equation*}
\sum^{k}_{i=1}b_i\geq k(k-n),
\end{equation*}
 for $1\leq k< n$ with equality when $k=n$.
\end{corollary}

Various results for imbalances in simple digraphs and oriented graphs can be found in \cite{3}, \cite{4}.

\section{Imbalances in $r$-graphs}

A multigraph is a graph from which multi-edges are not removed, and
which has no loops \cite{Gross2004}.  If $r \geq 1$ then an $r$-digraph (shortly
 $r$-graph)  is an orientation of a multigraph that is without loops and contains at most $r$ edges between the elements of any
pair of distinct vertices. Clearly 1-digraph is an oriented graph. Let $D$ be an $f$-digraph with vertex set $V=\{v_{1},v_{2},\ldots,v_{n}\}$, and let
$d_{v}^{+}$ and $d_{v}^{-}$ respectively denote the outdegree and indegree of vertex $v$. Define $b_{v_{i}}$ (or simply $b_{i}$) $=d_{v_{i}}^{+}-d_{u_{i}}^{-}$ as
imbalance of $v_{i}$. Clearly, $-r(n-1)\leq b_{v_{i}}\leq r(n-1)$. The imbalance sequence of $D$ is formed by listing the vertex imbalances in
non-decreasing order.

We remark that $r$-digraphs are special cases of $(a,b)$-digraphs containing at least $a$ and at most $b$ edges between the elements of any pair of vertices. Degree sequences of $(a,b)$-digraphs are studied in \cite{5,6}.

Let $u$ and $v$ be distinct vertices in $D$. If there are $f$ arcs directed from $u$ to $v$ and $g$ arcs directed from $v$ to $u$, we denote this by
$u(f-g)v$, where $0\leq f,g, f+g \leq r$.

A double in $D$ is an induced directed subgraph with two vertices $u,$ and $v$ having the form $u(f_1 – f_2)v$, where $1 \leq f_1$, $ f_2 \leq r,$ and $1 \leq f_1 + f_2 \leq r,$ and $f_1$ is the number of arcs directed from $u$ to $v,$ and $f_2$ is the number of arcs directed from $v$ to $u$. A triple in $D$ is an induced subgraph with tree vertices $u$, $v$, and $w$ having the form $u(f_1 – f_2)v(g_1 – g_2)w(h_1 – h_2)u$, where  $1 \leq f_1$, $f_2$, $g_1$, $g_2$,  $h_1$,  $h_2 \leq r,$ and $1 \leq f_1 + f_2$, $g_1 + g_2 $, $ h_1 + h_2 \leq r,$ and the meaning of $f_1$, $f_2$, $ g_1$, $ g_2$, $ h_1$, $ h_2$ is similar to the meaning in the definition of doubles. An oriented triple in $D$ is an
induced subdigraph with three vertices. An oriented triple is said to be transitive if it is of the form $u(1-0)v(1-0)w(0-1)u$, or
$u(1-0)v(0-1)w(0-0)u$, or $u(1-0)v(0-0)w(0-1)u$, or $u(1-0)v(0-0)w(0-0)u$, or $u(0-0)v(0-0)w(0-0)u$, otherwise it is
intransitive. An $r$-graph is said to be transitive if all its oriented triples are transitive. In particular, a triple $C$ in an $r$-graph is transitive if every oriented triple of $C$ is transitive.

The following observation can be easily established and is analogues to Theorem 2.2 of Avery \cite{1}.\\

\begin{lemma}
If $\,D_{1}$ and $D_{2}$ are two $r$-graphs with   same imbalance sequence, then $D_{1}$ can be transformed to $D_{2}$
  by successively transforming (i) appropriate oriented triples in one of the following ways, either (a) by changing the intransitive oriented triple
  $u(1-0)v(1-0)w(1-0)u$ to a transitive oriented triple  $u(0-0)v(0-0)w(0-0)u$, which has the same imbalance sequence or vice
  versa, or (b) by changing the intransitive oriented triple $u(1-0)v(1-0)w(0-0)u$ to a transitive oriented triple
  $u(0-0)v(0-0)w(0-1)u$, which has the same imbalance sequence or vice versa; or (ii) by changing a double $u(1-1)v$ to a double $u(0-0)v$,
  which has the same imbalance sequence or vice versa.
\end{lemma}

The above observations lead to the following result.

\begin{theorem} Among all $r$-graphs with given imbalance sequence, those with the fewest arcs are transitive.
\end{theorem}

\begin{proof}
Let $B$ be an imbalance sequence and let $D$ be a realization of $B$ that is not transitive. Then $D$ contains an
intransitive oriented triple. If it is of the form $u(1-0)v(1-0)w(1-0)u$, it can be transformed by operation \emph{i(a)} of Lemma 3 to a transitive oriented
triple $u(0-0)v(0-0)w(0-0)u$ with the same imbalance sequence and three arcs fewer. If $D$ contains an intransitive
oriented triple of the form $u(1-0)v(1-0)w(0-0)u$, it can be transformed by operation \emph{i(b)} of Lemma 3 to a transitive oriented
triple $u(0-0)v(0-0)w(0-1)u$ same imbalance sequence but one arc fewer. In case $D$ contains both types of intransitive
oriented triples, they can be transformed to transitive ones with certainly lesser arcs. If in $D$ there is a double $u(1-1)v$, by
operation \emph{(ii)} of Lemme 4, it can be transformed to $u(0-0)v$, with same imbalance sequence but two arcs fewer.
\end{proof}

The next result gives necessary and sufficient conditions for a sequence of integers to be the imbalance sequence of some $r$-graph.

\begin{theorem}
 A sequence $B=[b_{1},b_{2},\ldots,b_{n}]$ of   integers in non-decreasing order is an imbalance sequence of an $r$-graph if and only if
\begin{equation}
\sum^{k}_{i=1}b_i\geq rk(k-n),
\end{equation}
with equality when $k=n$.
\end{theorem}

\begin{proof}
{\bf Necessity.} A multi subdigraph induced by $k$ vertices has a sum of imbalances $rk(k-n)$.\\

{\bf Sufficiency.} Assume that $B=[b_{1},b_{2},\ldots,b_{n}]$ be the sequence of integers in non-decreasing order satisfying conditions (1)
but is not the  imbalance sequence of any $r$-graph. Let this sequence  be chosen in such a way that n is the smallest possible and
$b_1$ is the least with that choice of $n$. We consider the following two cases.

{\bf Case (i).} Suppose equality in (1) holds for some $k\leq n$, so that
\begin{equation*}
\sum_{i=1}^{k}b_{i}= rk(k-n),
\end{equation*}
for $1\leq k <n$.

By minimality of $n$, $B_{1}=[b_{1},b_{2},\ldots, b_{k}]$  is the imbalance sequence of some $r$-graph $D_1$ with vertex set, say $V_{1}$.
Let $B_{2}=[b_{k+1},b_{k+2},\ldots,b_{n}]$.\\
Consider,\\
\begin{equation*}
\begin{split}
\sum_{i=1}^{f}b_{k+i}&=\sum_{i=1}^{k+f}b_{i}-\sum_{i=1}^{k}b_{i}\\&\geq
 r(k+f)[(k+f)-n]-rk(k-n)\\&=r(k_{2}+kf-kn+fk+f_{2}-fn-k_{2}+kn)\\&\geq
 r(f_{2}-fn)\\&=rf(f-n),
\end{split}
\end{equation*}
for $1\leq f\leq n-k$, with equality when $f=n-k$. Therefore, by the minimality for $n$, the sequence $B_{2}$ forms the imbalance sequence
of some $r$-graph $D_{2}$ with vertex set, say $V_{2}$. Construct a new $r$-graph $D$ with vertex set as follows.\\
\indent Let $V=V_{1}\cup V_2$ with, $V_1\cap V_2=\phi$ and  the arc set containing those arcs which are in $D_1$ and $D_2$. Then we obtain
the $r$-graph $D$ with the imbalance sequence $B$, which is a contradiction.\\
{\bf Case (ii).} Suppose that the strict inequality holds in (1) for some $k<n$, so that\\
\begin{equation*}
\sum_{i=1}^{k}b_{i}>rk(k-n),
\end{equation*}
for $1\leq k<n$. Let $B_1=[b_{1}-1,b_{2},\ldots,b_{n-1},b_{n}+1]$, so that $B_1$ satisfy the conditions (1). Thus by the minimality of
$b_1$, the sequences $B_1$ is the imbalances sequence of some $r$-graph $D_{1}$ with vertex set, say $V_{1})$. Let $b_{v_{1}}=b_{1}-1$
and $b_{v_{n}}=a_{n}+1$. Since $b_{v_{n}}>b_{v_{1}}+1$, there exists a vertex $v_p\in V_1$ such that $v_{n}(0-0)v_{p}(1-0)v_1$, or
$v_{n}(1-0)v_{p}(0-0)v_1$, or $v_{n}(1-0)v_{p}(1-0)v_1$, or $v_{n}(0-0)v_{p}(0-0)v_1$, and if these  are changed to $v_{n}(0-1)v_{p}(0-0)v_1$, or
$v_{n}(0-0)v_{p}(0-1)v_1$, or $v_{n}(0-0)v_{p}(0-0)v_1$, or $v_{n}(0-1)v_{p}(0-1)v_1$ respectively, the result is an $r$-graph
with imbalances sequence $B$, which is again a contradiction. This proves the result.
\end{proof}

Arranging the imbalance sequence in non-increasing order, we have the following observation.

\begin{corollary}
A sequence $B=[b_{1},b_{2},\ldots,b_{n}]$ of integers with $b_{1}\geq b_{2}\geq \ldots \geq b_{n}$ is an
  imbalance sequence of an $r$-graph if and only if
\begin{equation*}
\sum_{i=1}^{k}b_{i}\leq rk(n-k),
\end{equation*}
for $1\leq k\leq n$, with equality when $k=n$.
\end{corollary}

The converse of an $r$-graph $D$ is an  $r$-graph $D^\prime$, obtained by reversing orientations of all arcs of $D$. If
$B=[b_{1},b_{2},\ldots,b_{n}]$ with $b_{1}\leq b_{2}\leq \ldots \leq b_{n}$ is the imbalance sequence of an $r$-graph $D$, then
$B^{\prime}=[-b_{n},-b_{n-1},\ldots,-b_{1}]$ is  the imbalance sequence of $D'$.

The next result gives lower and upper bounds for the imbalance $b_{i}$ of a vertex $v_{i}$ in an $r$-graph $D$.

\begin{theorem}
If $B=[b_{1},b_{2},\ldots,b_{n}]$ is an imbalance sequence of an $r$-graph $D$, then for each $i$
\begin{equation*}
r(i-n)\leq b_{i}\leq r(i-1).
\end{equation*}
\end{theorem}

\begin{proof}
Assume to the contrary that $b_{i}<r(i-n)$, so that for $k<i$,
\begin{equation*}
b_{k}\leq b_{i}< r(i-n).
\end{equation*}
That is,
\begin{equation*}
b_{1}< r(i-n),b_{2}< r(i-n),\ldots,b_{i}< r(i-n).
\end{equation*}
Adding these inequalities, we get
\begin{equation*}
\sum_{k=1}^{i}b_{k}< ri(i-n),
\end{equation*}
which contradicts Theorem 3.\\
\indent Therefore, $r(i-n)\leq b_{i}$.\\
\indent The second inequality is dual to the first. In the converse
$r$-graph with imbalance sequence
$B=[b_{1}^{\prime},b_{2}^{\prime},\ldots,b_{n}^{\prime}]$ we have, by
the first inequality
\begin{equation*}
\begin{split}
b_{n-i+1}^{\prime}&\geq r[(n-i+1)-n]\\&=r(-i+1).
\end{split}
\end{equation*}
\indent Since $b_{i}=-b_{n-i+1}^{\prime}$, therefore
\begin{equation*}
b_{i}\leq -r(-i+1)=r(i-1).
\end{equation*}
\indent Hence, $b_{i}\leq r(i-1)$.
\end{proof}

Now we obtain the following inequalities for imbalances in $r$-graphs.

\begin{theorem}
If $B=[b_{1},b_{2},\ldots,b_{n}]$ is an imbalance sequence of an $r$-graph with $b_{1}\geq b_{2}\geq
  \ldots \geq b_{n}$, then
\begin{equation*}
\sum_{i=1}^{k}b_{i}^{2}\leq \sum_{i=1}^{k}(2rn-2rk-b_{i})^{2},
\end{equation*}
for $1\leq k\leq n$ with equality when $k=n$.
\end{theorem}

\begin{proof}
By Theorem 3, we have for $1\leq k\leq n$ with equality when $k=n$
\begin{equation*}
rk(n-k)\geq \sum_{i=1}^{k}b_{i},
\end{equation*}
implying
\begin{equation*}
\sum_{i=1}^{k}b_{i}^{2}+2(2rn-2rk)rk(n-k)\geq \sum_{i=1}^{k}b_{i}^{2}+2(2rn-2rk)\sum_{i=1}^{k}b_{i},
\end{equation*}
from where
\begin{equation*}
\sum_{i=1}^{k}b_{i}^{2}+k(2rn-2rk)^{2}-2(2rn-2rk)\sum_{i=1}^{k}b_{i}\geq \sum_{i=1}^{k}b_{i}^{2},
\end{equation*}
and so we get the required
\begin{equation*}
\begin{split}
b_{1}^{2}+b_{2}^{2}+\ldots
+b_{k}^{2}&+(2rn-2rk)^{2}+(2rn-2rk)^{2}+\ldots+(2rn-2rk)^{2}\\&-2(2rn-2rk)b_{1}-2(2rn-2rk)b_{2}-\ldots
-2(2rn-2rk)b_{k}\\&\geq \sum_{i=1}^{k}b_{i}^{2},
\end{split}
\end{equation*}
or
\begin{equation*}
\sum_{i=1}^{k}(2rn-2rk-b_{i})^{2}\geq \sum_{i=1}^{k}b_{i}^{2}.
\end{equation*}
\end{proof}

\indent The set of distinct imbalances of vertices in an $r$-graph
is called its imbalance set. The following result gives the existence
of an $r$-graph with a given imbalance set. Let $(p_1,p_2,\ldots,p_m,q_1,q_2,\ldots,q_n)$ denote the greatest common divisor of $p_1,p_2,\ldots,p_n,q_1,q_2,\ldots,q_n$.

\begin{theorem}
If $P=\{p_{1},p_{2},\ldots,p_{m}\}$ and  $Q=\{-q_{1},-q_{2},\ldots,-q_{n}\}$  where  \linebreak $p_{1},p_{2},\ldots,p_{m},q_{1},q_{2},\ldots,q_{n}$ are positive
  integers such that $p_{1}<p_{2}<\ldots <p_{m}$ and $q_{1}<q_{2}<\ldots <q_{n}$ and $(p_{1},p_{2},\ldots,p_{m},q_{1},q_{2},\ldots,q_{n})=t$, $1\leq
  t\leq r$, then there exists an $r$-graph with imbalance set $P\cup Q$.
\end{theorem}

\begin{proof}
Since  $(p_{1},p_{2},\ldots,p_{m},q_{1},q_{2},\ldots,q_{n})=t$, $1\leq  t\leq r$, there exist positive integers $f_{1},f_{2},\ldots,f_{m}$
  and $g_{1},g_{2},\ldots,g_{n}$ with $f_{1}<f_{2}<\ldots<f_{m}$ and $g_{1}<g_{2}<\ldots<g_{n}$ such that
\begin{equation*}
p_{i}=tf_{i}
\end{equation*}
for $1\leq i\leq m$ and
\begin{equation*}
q_{i}=tg_{i}
\end{equation*}
for $1\leq j\leq n$.\\
\indent We construct an $r$-graph $D$ with vertex set $V$ as
follows.\\
\indent Let
\begin{equation*}
V=X_{1}^{1}\cup X_{2}^{1}\cup \ldots \cup
X_{m}^{1}\cup X_{1}^{2}\cup X_{1}^{3}\cup \ldots \cup X_{1}^{n}\cup Y_{1}^{1}\cup Y_{2}^{1}\cup \ldots \cup
Y_{m}^{1}\cup Y_{1}^{2}\cup Y_{1}^{3}\cup \ldots \cup Y_{1}^{n},
\end{equation*}
with $X_{i}^{j}\cap X_{k}^{l}=\phi$, $Y_{i}^{j}\cap Y_{k}^{l}=\phi$,
$X_{i}^{j}\cap Y_{k}^{l}=\phi$ and\\
\indent $|X_{i}^{1}|=g_{1}$, for all $1\leq i\leq m$,\\
\indent $|X_{1}^{i}|=g_{i}$, for all $2\leq i\leq n$,\\
\indent $|Y_{i}^{1}|=f_{i}$, for all $1\leq i\leq m$,\\
\indent $|Y_{1}^{i}|=f_{1}$, for all $2\leq i\leq n$.\\
\indent Let there be $t$ arcs directed from every vertex of
$X_{i}^{1}$ to each vertex of $Y_{i}^{1}$, for all $1\leq i\leq m$ and
let there be $t$ arcs directed from every vertex of
$X_{1}^{i}$ to each vertex of $Y_{1}^{i}$, for all $2\leq i\leq n$ so
that we obtain the $r$-graph $D$ with imbalances of vertices as
under.\\
\indent For $1\leq i\leq m$, for all $x_{i}^{1}\in X_{i}^{1}$
\begin{equation*}
b_{x_{i}^{1}}=t|Y_{i}^{1}|-0=tf_{i}=p_{i},
\end{equation*}
for $2\leq i\leq n$, for all $x_{1}^{i}\in X_{1}^{i}$
\begin{equation*}
b_{x_{1}^{i}}=t|Y_{1}^{i}|-0=tf_{1}=p_{1},
\end{equation*}
for $1\leq i\leq m$, for all $y_{i}^{1}\in Y_{i}^{1}$
\begin{equation*}
b_{y_{i}^{1}}=0-t|X_{i}^{1}|=-tg_{i}=-q_{i},
\end{equation*}
and for $2\leq i\leq n$, for all $y_{1}^{i}\in Y_{1}^{i}$
\begin{equation*}
b_{y_{1}^{i}}=0-t|X_{1}^{i}|=-tg_{i}=-q_{i}.
\end{equation*}
\indent Therefore imbalance set of $D$ is $P\cup Q$.
\end{proof}

\section*{Acknowledgement} The research of the fourth author was supported by the project T\'AMOP-4.2.1/B-09/1/KMR-2010-0003 of
E\"otv\"os Lor\'and University.

The authors are indebted for the useful remarks of the unknown referee.

\bigskip
\rightline{\emph{Received: August 23, 2010}}     


\begin{thebibliography}{90}

\bibitem{1} P. Avery, Score sequences of oriented graphs, \textit{J. Graph Theory}, {\bf 15} (1991), 251--257.

\bibitem{Gross2004}  J. L. Gross, J. Yellen, \emph{Handbook of graph theory}, CRC Press, London/New York, 2004.


\bibitem{5} A. Iv\'anyi, Reconstruction of complete interval tournaments,
\textit{Acta Univ. Sapientiae, Inform.},
\textbf{1} (2009), 71--88.

\bibitem{6} A. Iv\'anyi, Reconstruction of complete interval tournaments II,
\textit{Acta Univ. Sapientiae, Math.},
\textbf{2} (2010), 47--71.

\bibitem{2} D. Mubayi, T. G. Will and D. B. West, Realizing degree imbalances in directed graphs, \textit{Discrete Math.},
{\bf 239} (2001), 147--153.


\bibitem{3} S. Pirzada, T. A. Naikoo and N. A. Shah, Imbalances in oriented tripartite graphs,
\textit{Acta Math. Sinica}, (2010) (to appear).

\bibitem{4} S. Pirzada, On imbalances in digraphs, \textit{Kragujevac J. Math.}, {\bf 31}  (2008), 143--146.
\end{thebibliography}
\end{document}